\documentclass[11pt]{amsart}
\usepackage{amssymb,amsmath,amsthm}
\usepackage{color}

\numberwithin{equation}{section}

\theoremstyle{theorem}
\newtheorem{theorem}{Theorem}[section]
\newtheorem{lemma}[theorem]{Lemma}

\newtheorem{conjecture}[theorem]{Conjecture}

\theoremstyle{definition}

\newtheorem{remark}{Remark}

\newcommand{\rr}{{\mathbb R}}
\newcommand{\nn}{{\mathbb N}}

\newcommand{\al}{\alpha}
\newcommand{\be}{\beta}

\newcommand{\ep}{\epsilon}
\newcommand{\la}{\lambda}

\title{Cantor set arithmetic}
\author{Jayadev S. Athreya}
\address{Department of Mathematics, University of
Washington, Seattle, WA 98195 }
\email{jathreya@math.uiuc.edu}
\author{Bruce Reznick}
\address{Department of Mathematics, University of
Illinois at Urbana-Champaign, Urbana, IL 61801}
\email{reznick@math.uiuc.edu}
\author{Jeremy T. Tyson}
\address{Department of Mathematics, University of
Illinois at Urbana-Champaign, Urbana, IL 61801}
\email{tyson@math.uiuc.edu}
\thanks{J.S.A. was supported by NSF CAREER grant DMS-1559860 and NSF grants DMS-1069153, DMS-1107452, DMS-1107263 and DMS-1107367. B.R. was supported by
Simons Collaboration Grant 280987. J.T.T. was supported by NSF grants DMS-1201875 and DMS-1600650 and Simons Collaboration Grant 353627.}
\date{\today}
\begin{document}

\begin{abstract}
Every element $u$ of $[0,1]$ can be written in the form $u=x^2y$,
where $x,y$ are elements of the Cantor set $C$. In particular, every
real number between zero and one is the product of three elements of
the Cantor set. 
On the other hand the set of real numbers $v$ { that} can be written
in the form $v=xy$ with $x$ and $y$ in $C$ is a closed subset of
$[0,1]$ with Lebesgue measure strictly between $\tfrac{17}{21}$ and
$\tfrac89$. We also describe the structure of the quotient of $C$ by
itself, that is, the image of 
$C\times (C \setminus \{0\})$ under the function $f(x,y) = x/y$.
\end{abstract}

\maketitle
\section{Introduction.}

One of the first exotic mathematical objects encountered by the
post-calculus student is the Cantor set
\begin{equation}\label{Steinhaus}
C = \left\{\sum_{k= 1}^\infty \al_k3^{-k}, \al_k \in \{0,2\}\right\}.
\end{equation}
(See {Section 2} for
several equivalent definitions of $C$.) One of its most beautiful
properties is that
\begin{equation}\label{C+C}
C+C := \{x+y: x,y \in C\}
\end{equation}
is equal to $[0,2]$. (The whole interval is produced by adding dust to itself!)

The first published proof of \eqref{C+C} was by Hugo Steinhaus
\cite{steinhaus:sums} in 1917. The result was later rediscovered by
John Randolph in 1940 \cite{ran:cantor}. 

We remind the reader of the beautiful constructive proof of
\eqref{C+C}. It is enough to prove the containment $C+C \supset
[0,2]$. Given $u \in [0,2]$, consider the ternary representation for
$u/2$: 
\begin{equation}
\frac u2= \sum_{k= 1}^{\infty} \frac{\ep_k}{3^k}, \quad \ep_k \in \{0,1,2\}.
\end{equation}
Define pairs $(\al_k,\be_k)$ to be $(0,0), (2,0), (2,2)$ according to
whether $\ep_k = 0, 1, 2$, respectively, and define elements $x, y \in
C$ by
\begin{equation}
x = \sum_{k= 1}^{\infty} \frac{\al_k}{3^k}, \qquad
y = \sum_{k= 1}^{\infty} \frac{\be_k}{3^k}.
\end{equation}
Since $\al_k + \be_k = 2\ep_k$, $x+y = 2\cdot\frac u2 = u$.

While presenting this proof in a class, one of the authors 
(BR) 
wondered what would happen if addition were replaced by the other arithmetic
operations. Another author 
(JT) 
immediately pointed out that subtraction is easy, because of a symmetry of $C$:
\begin{equation*}
x = \sum_{k=1}^\infty \frac{\ep_k}{3^k} \in C \iff 1-x = \sum_{k=1}^\infty
\frac{2-\ep_k}{3^k} \in C.
\end{equation*}
Thus
\begin{equation*}
C-C := \{x-y: x,y \in C\} =  \{x-(1-z): x,z \in C\} = C+C -1 = [-1,1].
\end{equation*}
More generally, to understand the structure of linear combinations
$aC+bC$, $a,b\in\rr$, it suffices to consider the case $a=1$ and $0\le
b\le 1$. (If $a>b>0$, then $aC+bC=a(C+(b/a)C)$; the remaining cases
are left to the reader.) The precise structure of the linear
combination set 
$$
aC+bC
$$
was obtained by Paw{\l}owicz \cite{paw:cantor}, who extended an
earlier result by Utz \cite{utz:cantor}, published in 1951. Utz's
result states that  
\begin{equation}\label{eq:utz}
C + b C = [0,1+b] \qquad \mbox{for every $\tfrac13 \le b \le 3$.}
\end{equation}

Multiplication is trickier. The inclusion $C \subset [0,\frac13] \cup
[\frac23, 1]$ 
implies that any element of the interval $(\frac 13,\frac49)$ cannot
be written as a product 
of two elements from $C$. { Thus the measure of the product of $C$
  with itself is at most $\tfrac89$.} This paper grew out of a study
of multiplication on $C$. 

In this paper we will prove the following results.

\begin{theorem}
\label{maintheorem}
\
\begin{enumerate}
\item Every $u \in [0,1]$ can be written as $u = x^2y$ for some $x,y \in C$.
\item The set of quotients from $C$ can be described as follows:
\begin{equation}\label{E:quot}
\left\{ \frac xy: x, y \in C, y\neq 0\right\} = \bigcup_{m = -\infty}^\infty
\left[\frac 23 \cdot 3^m ,  \frac 32 \cdot 3^m\right].
\end{equation}
\item The set $\{xy:x,y \in C\}$ is a closed set with Lebesgue measure
  { strictly greater than} 
  $\frac{17}{21}$.
\end{enumerate}
\end{theorem}

In particular, part (1) of Theorem \ref{maintheorem} implies that
every real number in the interval $[0,1]$ is the product of three
elements of $C$. 

In words, \eqref{E:quot} says that each positive real number is a quotient of two elements of
the Cantor set if and only if either the left-most nonzero digit in
the ternary representation of $u$ is ``2,'' or the left-most nonzero digit
is ``1,'' but the first subsequent non-``1'' digit is ``0,'' not ``2.''

This paper is organized as follows. We begin in section \ref{sec:tools} with
several different descriptions of the Cantor set, and then the key tools, all of which
are accessible to { students in} a good undergraduate analysis class. As a warmup, we
use these tools to give a short proof of Utz's result \eqref{eq:utz}
in Section \ref{subsec:plus-and-minus}. Sections \ref{subsec:times},
\ref{subsec:divide} and \ref{subsec:times-again} are devoted to the
proofs of {parts (1), (2) and (3) of} Theorem~\ref{maintheorem},
respectively. Sprinkled throughout are relevant open questions. 
This article began as a standard research paper, but we realized that
many of our main results might be of interest to a wider audience. In
Section \ref{sec:final-remarks}, we discuss some of our other results,
which will be 
published elsewhere, with different combinations of co-authors.

\section{Tools.}\label{sec:tools}

We begin by recalling several equivalent and well-known definitions of the Cantor
set. See Fleron \cite{fl:cantor} and the references within for an excellent overview of the history of the Cantor set and the context in which several of these definitions first arose.

The standard ternary representation of a real number $x$ in $[0,1]$ is
\begin{equation}\label{E:ternrep}
x = \sum_{k=1}^{\infty} \frac {\al_k(x)}{3^k},\qquad  \al_k(x) \in \{0,1,2\}.
\end{equation}
This representation is unique, except for the {\it ternary rationals},
$\{ \frac m{3^n}, m,n \in \mathbb N\}$, which have two
ternary representations. Supposing $\al_n > 0$ and $m \neq 0 \mod
3$, so that $\al_n \in \{1,2\}$ below, we have
\begin{equation}\label{E:ternexc}
\begin{split}
\frac m{3^n}
&= \sum_{k=1}^{n-1} \frac {\al_k}{3^k} + \frac{\al_n}{3^n} + \sum_{k=n+1}^{\infty} \frac {0}{3^k}\\
&= \sum_{k=1}^{n-1} \frac {\al_k}{3^k} + \frac{\al_n-1}{3^n} + \sum_{k=n+1}^{\infty} \frac {2}{3^k}.
\end{split}
\end{equation}

The {\it Cantor set} $C$ consists those $x \in [0,1]$ { admitting} a ternary representation as in \eqref{E:ternrep} with
$\al_k(x) \in \{0,2\}$ for all $k$. Note that $C$ also contains those ternary rationals as in \eqref{E:ternexc}  whose final digit is ``1.'' These may be
transformed as above into a representation in which $\al_n(x) = 0$  and $\al_k(x) = 2$ for $k > n$.
As noted earlier, $x \in C$ if and only if $1-x \in C$. Further, if $k \in \mathbb N$, then
\begin{equation}
x \in C  \implies 3^{-k}x \in C.
\end{equation}
This definition arises in dynamical systems, as the Cantor set $C$ can be viewed as an invariant set for the map $x \mapsto 3x \bmod 1$, or equivalently, as the image of an invariant set $C'$ for the one-sided shift map $\sigma$ acting on $\Omega = \{0,1,2\}^\nn$. Given a sequence $\omega = (\omega_n)_{n=1}^\infty = (\omega_1,\omega_2,\ldots)$ in $\Omega$,
$$
\sigma\omega = (\omega_2,\omega_3,\ldots)\,.
$$
Letting $C' = \{0,2\}^\nn \subset \Omega$, we realize the Cantor set $C$ as the image of $C'$ under the coding map $T:\Omega \to [0,1]$ given by
$$
T(\omega) = \sum_{n=1}^\infty \omega_n 3^{-n}.
$$

We now present the usual ``middle-third'' definition of the Cantor
set: Define $C_n = \{x: \al_k(x) \in \{0,2\}, 1 \le k \le n\}$, which is
a union of $2^n$ closed intervals of length $3^{-n}$, written as
\begin{equation}
C_n = \bigcup_{i=1}^{2^n} I_{n,i}.
\end{equation}
The left-hand endpoints of the $I_{n,i}$'s comprise the set
\begin{equation}
\left \{\sum_{k=1}^{n} \frac{\ep_k}{3^k}: \ep_k \in \{0,2\}\right\}.
\end{equation}
The right-hand endpoints  have ``1'' as their final nonzero
ternary digit when written as a finite ternary expansion.

The  more direct definition of $C$ is as a nested intersection of
closed sets:
\begin{equation}
C = \bigcap_{n=1}^\infty C_n;\qquad  C_1 \supset C_2 \supset C_3 \supset \cdots.
\end{equation}
This definition is standard in fractal geometry, where the Cantor set $C$ is seen as the invariant set for the pair of contractive linear mappings $f_1(x) = \tfrac13 x$ and $f_2(x) = \tfrac13 x + \tfrac23$ acting on the real line. That is, $C$ is the unique nonempty compact set that is fully invariant under $f_1$ and $f_2$:
$$
C = f_1(C) \cup f_2(C).
$$

Observe that each ``parent'' interval $I_{n,i} = [a,a+\frac 1{3^n}]$ in $C_n$ has two ``child'' intervals
\begin{equation}
I_{n+1,2i-1} = \left[ a,a+\tfrac 1{3^{n+1}}\right],\qquad
I_{n+1,2i} = \left[a+\tfrac2{3^{n+1}},a+\tfrac3{3^{n+1}}\right]
\end{equation}
in $C_{n+1}$, and $C_{n+1}$ is the union of all  children
intervals whose parents are in $C_n$.

It is useful to introduce the following notation to
represent the omission of the middle third:
\begin{equation}
I = [a,a+3t] \implies \ddot I = [a,a+t] \cup [a+2t,a+3t].
\end{equation}
 Using this notation,
\begin{equation}
C_{n+1} = \bigcup_{i=1}^{2^{n+1}} I_{n+1,i} = \bigcup_{i=1}^{2^n} \ddot I_{n,i}.
\end{equation}

It will also be useful, for studying products and quotients, to give a third
definition. Let $\widetilde C$ = $C \cap [1/2,1] = C \cap
[2/3,1] = \frac 23 + \frac 13\cdot C$. Then by examining
the smallest $k$ for which $\ep_k = 2$, we see that
\begin{equation}
C = \{0\} \cup \bigcup_{k=0}^{\infty} 3^{-k}\widetilde C.
\end{equation}
Similarly, let $\widetilde C_n$ = $C_n \cap [1/2,1]$, consisting of
$2^{n-1}$ closed intervals of length $3^{-n}$:
\begin{equation}
\widetilde C_n = \bigcup_{i=2^{n-1}+1}^{2^n} I_{n,i}.
\end{equation}
Then
\begin{equation}
\widetilde C = \bigcap_{n=1}^\infty \widetilde C_n.
\end{equation}
And, analogously,
\begin{equation}
\widetilde C_{n+1} = \bigcup_{i=2^{n}+1}^{2^{n+1}} I_{n+1,i} =
\bigcup_{i=2^{n-1}+1}^{2^n} \ddot I_{n,i}.
\end{equation}
By looking at the left-hand endpoints, we see that
each interval $I_{n,i} \neq [0,\frac 1{3^n}]$ can be written as
$\frac 1{3^{n-k}} I_{k,j}$ for some $I_{k,j} \in \widetilde C_k$; hence
\begin{equation}
\begin{gathered}
C_n = \left[0,\frac 1{3^n}\right] \cup \bigcup_{k=1}^{n} \frac
1{3^{n-k}} \widetilde C_k.
\end{gathered}
\end{equation}

The keys to our method lie in  two lemmas which might appear on a
first serious analysis exam. (Please do not send such exams to the authors!)

\begin{lemma}\label{lemma1}
Suppose  $\{K_i\}
\subset \mathbb R$ are nonempty compact sets, $K_1
\supseteq K_2 \supseteq K_3 \supseteq \cdots$, and $K = \cap K_i$.

\noindent (i)  If $(x_j) \to x$, $x_j \in K_j$, then $x \in K$.

\noindent (ii) If $F:\mathbb R^m
\to \mathbb R$ is continuous, then $F(K^{m}) = \cap F(K_i^{m})$.
\end{lemma}

\begin{proof}
(i). If $x \notin K$, then $x \notin K_r$ for some
$r$. Since $K_r^c$ is
open, there exists $\ep > 0$ so that $(x - \ep, x+ \ep) \subseteq K_r^c
\subseteq K_{r+1}^c \cdots$ and hence
$|x_j - x| \ge \ep$ for $j \ge r$, a contradiction to $x_j \to x$.

(ii). Since $K \subseteq K_i$, { we have} $F(K^{m}) \subseteq \cap F(K_i^{m})$. Conversely, suppose $u \in \cap F(K_i^{m})$. We need to find
$x \in K^{m}$ such that $F(x) = u$. For each $i$, choose $x_i = (x_{i,1}, \dots, x_{i,m}) \in K_i^m$ so that $F(x_i) = u$. Since $K_1^{m}$ is compact,
the Bolzano--Weierstrass theorem implies { that} the sequence $(x_i)$ has a convergent subsequence
$x_{r_j} = (x_{r_j,1}, \dots x_{r_j,m}) \to y = (y_1,\dots, y_m)$.
Applying (i) to the subsequence $K_{r_1} \supseteq K_{r_2} \supseteq
K_{r_3} \supseteq \cdots$, we see that each $y_k \in K$ and
since $F$ is continuous, $F(y) = u$, as desired.
\end{proof}

If we perform the middle-third construction with an initial
interval of $[a,b]$, it is easy to see that the limiting object is a
translate of the Cantor set, specifically $C_{a,b}:= a +(b-a)C$.

\begin{lemma}\label{lemma2}
Suppose $F: \mathbb R^m \to \mathbb R$ is continuous, and suppose
that for every  choice of disjoint or identical subintervals $I_k
\subset [a,b]$ of common length,
\begin{equation}
F(I_1,\dots,I_m) = F(\ddot I_1,\dots, \ddot I_m).
\end{equation}
Then $F(C_{a,b}^{m}) = F([a,b]^{m})$.
\end{lemma}

\begin{proof}
We prove the result for $[a,b] = [0,1]$; the result follows generally by
composing $F$ with an appropriate linear function. Let
\begin{equation}
C_k = \bigcup_{j=1}^{2^k} I_{k,j},
\end{equation}
where each interval $I_{k,j}$  has length $3^{-k}$.
It follows that
\begin{equation}
F(C_k^{m}) = \bigcup_{1\le j_1,\dots,j_m\le 2^k} F(I_{k,j_1},\dots, I_{k,j_m}),
\end{equation}
where for each pair $(\ell,\ell')$, $I_{k,j_\ell}$ and $I_{k,j_\ell'}$
are either identical or disjoint.
Since
\begin{equation}
C_{k+1} = \bigcup_{j=1}^{2^k} \ddot I_{k,j},
\end{equation}
the hypothesis implies that  $F(C_k^{m}) = F(C_{k+1}^{m})$, and
the result then follows from Lemma \ref{lemma1}(ii).
\end{proof}

We only apply Lemma \ref{lemma2}  in the cases that $m=2$ and $[a,b] =
[0,1]$ and $[\frac23,1]$. (In the latter case, when $F(x,y) = xy$ or $x/y$,
it is helpful to have control of the ratio $x/y$.)

\section{Arithmetic on the Cantor set.}

\subsection{Addition and subtraction}\label{subsec:plus-and-minus}

Sums and differences of Cantor sets have been widely studied in
connection with dynamical systems. In this section we give a brief
proof of the following result of Utz \cite{utz:cantor}. Further
information about sums of Cantor sets can be found in \cite{chm:sums} and
\cite{chm:sums2}. 

\begin{theorem}[Utz]
If $\la \in [\frac 13,3]$, then every element $u$ in $[0,1+\la]$ can be written in the form $u = x + \la y$ for $x,y \in C$.
\end{theorem}

We include this proof in order to introduce the main ideas in the proof of Theorem \ref{maintheorem} in a simpler context. The key tool is Lemma \ref{lemma2}.

Let  $f_{\la}(x,y) = x + \la y$; we wish to show that $f_{\la}(C^{2}) = [0,1+\la]$. Observe that $C + \la C = \la(C + \la^{-1}C)$ for $\la \neq 0$, so it suffices to consider $\tfrac13 \le \la \le 1$.

\begin{proof}
We apply Lemma \ref{lemma2} and show that for any two closed
intervals $I_1,I_2$ of the same length in $[a,b]$, $f_{\la}(I_1,I_2) =
f_{\la}(\ddot I_1,\ddot I_2)$. For clarity, we write $I_1 =
[r,r+3t]$, $I_2 = [s,s+3t]$, and $w = r + \la s$,  so that
$f_{\la}(I_1,I_2) = [w, w+ 3(1+\la)t]$.

Observe that $\ddot I_1 = [r,r+t]\cup[r+2t,r+3t]$
and  $\ddot I_2 = [s,s+t]\cup[s+2t,s+3t]$, so
\begin{equation}\begin{split}
f_{\la}(\ddot I_1,\ddot I_2)
&= \left(\ [w,w+(1+\la)t] \cup [w+2\la t,w+(1+3\la)t]\ \right) \\
&\cup \left(\ [w+2t,w+(3+\la)t] \cup [w+(2+2\la t,w+(3+3\la)t]\ \right) \,.
\end{split}\end{equation}
Since $\la \le 1$, $1+\la \ge 2\la$ and  $3+\la \ge 2+ 2\la$, the pairs of intervals coalesce into
\begin{equation}
f_{\la}(\ddot I_1,\ddot I_2) = [w,w+(1+3\la)t]) \cup [w+2t,w+(3+3\la)t].
\end{equation}
Since $\la \ge \frac 13$, { we have} $2 \le 1+3\la$. { Hence} $f_{\la}(\ddot I_1,\ddot I_2) =f_{\la}(I_1,I_2)$, completing the proof.
\end{proof}

\begin{remark}
Unlike the folklore proof for $\la = 1$, there seems to be no obvious
algorithmic proof, save for  $\la = \frac 13$. In this case, suppose
$u \in [0,\frac 43]$. If $u = \frac 43$,
then $u = 1 + \frac 13\cdot 1 \in C + \frac 13 C$. If $u < \frac 43$,
then we can write $u = x + y$, $x, y \in C$, and assume $x\ge
y$. Since $y \le \frac u2 <
\frac 23$ is in $C$,  $y \le \frac 13$, hence $y = \frac 13 z$, $z \in C$. This
produces the desired construction.
\end{remark}

The case of subtraction, {that is, the case of} $f_{\la}$ when $\la < 0$, is easily handled.

\begin{theorem}
If $\be = -\la <0$, then
\begin{equation}
f_{\be}(C^{2}) =  -\la + f_{\la}(C^{2}).
\end{equation}
\end{theorem}

\begin{proof}
If $\be < 0$, $x,y \in C$, we have
\begin{equation}
x + \be y = x + \be(1-z) = -\la  + x + \la z
\end{equation}
for $x$ and $z$ in $C$.
\end{proof}

\begin{remark}\label{rem:sums}
Arithmetic sums of Cantor sets and more general compact sets have been studied intensively. For instance, Mendes and Oliveira \cite{mo:topological} discuss the topological structure of sums of Cantor sets, while Schmeling and Shmerkin \cite{SS:dimension} characterize those nondecreasing sequences $0\le d_1 \le d_2\le d_3\le \cdots \le 1$ { that} can arise as the sequence of Hausdorff dimensions of iterated sumsets $A,A+A,A+A+A,\ldots$ for a compact subset $A$ of $\rr$. Recent work of Gorodetski and Northrup \cite{gn:sums} involves the Lebesgue measure of sumsets of Cantor sets and other compact subsets of the real line.
We refer the interested reader to these papers and the references therein for more information.
\end{remark}

\subsection{Multiplication}\label{subsec:times}

We let $f(x,y) = x^2 y$ and shall show that $f(C^2) = [0,1]$. We begin
by showing that it suffices to consider $f(\widetilde C^2)$.

\begin{lemma}
If  $f(\widetilde C^2) = [\frac 8{27},1]$, then  $f(C^2) = [0,1]$.
\end{lemma}

\begin{proof}
Suppose $u \in [0,1]$. If $u = 0$, then $u = 0^2\cdot 0$. If $u > 0$,
then there exists a
unique integer $r \ge 0$ so that $u = 3^{-r}v$, where $v \in (\frac
13,1]$. Since $\frac 8{27} < \frac 13$, $v = x^2y$ for $x,y\in
\widetilde C \subset C$, and since $x, 3^{-r}y \in C$,
$u = x^2 (3^{-r}y)$ is the desired representation.
\end{proof}

Accordingly, we confine our attention to $\widetilde C$.
\begin{lemma}
If $I = [a,a+3t]$ and $J = [b,b+3t]$ are in
$[\frac 23, 1]$, then $f(I,J) = f(\ddot I, \ddot J)$.
\end{lemma}
\begin{proof}
We first define
\begin{equation}
\begin{gathered}
\
[a^2b, (a+t)^2 (b + t)]=: [u_1,v_1]; \\
[a^2(b+2t), (a+t)^2 (b + 3t)]=: [u_2,v_2]; \\
[(a+2t)^2b, (a+3t)^2 (b + t)]=: [u_3,v_3]; \\
[(a+2t)^2(b+2t), (a+3t)^2 (b +3 t)]=: [u_4,v_4].
\end{gathered}
 \end{equation}
Evidently, $f(I,J) = [u_1,v_4]$, and also, $u_1 < u_2, v_1 < v_2$, and $u_3
< u_4, v_3 < v_4$. If we can first show that $v_1 > u_2$ and $v_3 > u_4$,
then $[u_1,v_1] \cup  [u_2,v_2] = [u_1, v_2]$ and
$[u_3,v_3] \cup  [u_4,v_4] = [u_3, v_4]$. Second, since $u_1 < u_3$ and $v_2 <
v_4$, if we can show that $v_2 > u_3$, then $[u_1,v_2] \cup  [u_3,v_4]
= [u_1, v_4]$, and the proof will be complete.

We compute
\begin{equation}\label{eq:36}
\begin{gathered}
v_1 - u_2 = a(2b-a)t + (2a + b) t^2 + t^3; \\
v_2 - u_3 = a(3a - 2b)t + (6a - 3b)t^2 + 3t^3; \\
v_3 - u_4  = a(2b - a)t + (5b - 2a) t^2 + t^3.
\end{gathered}
\end{equation}
Since $2b - a \ge 2\cdot \frac 23 - 1 > 0$, $3a-2b \ge 3\cdot \frac
23 - 2 \ge 0$,   $6a - 3b \ge 6 \cdot \frac 23 -3 > 0$, and $5b-2a \ge
5\cdot \frac 23 - 2 > 0$, each of the quantities in \eqref{eq:36} is positive and
the proof is complete.
\end{proof}
\begin{theorem}\label{mult}
\begin{equation}
f(\widetilde C^2) = [\tfrac 8{27},1].
\end{equation}
\end{theorem}
\begin{proof}
Apply Lemma \ref{lemma2}, noting that $f([\frac 23,1]^2) =
[\tfrac 8{27},1]$.
\end{proof}

\begin{remark}
Observe that $u \in [0,1]$, written as $x^2y$, where $x,y \in C$, is also
$x\cdot  x \cdot y$, so this implies that every element in
$[0,1]$ is a product of three elements from the Cantor set. (This can
also be proved in a more ungainly way, by taking $m=3$ in Lemma
\ref{lemma2}  with $f(x_1,x_2,x_3) = x_1x_2x_3$.)
\end{remark}

\begin{remark}
We do not know an algorithm for expressing $u \in [0,1]$ in the form
of $x^2y$ or $x_1x_2x_3$ as a product of elements of the Cantor set.
\end{remark}

\begin{remark}
More generally, one can look at $f_{a,b}(C^2)$ where $f_{a,b}(x,y) :=
x^ay^b$. Since $u = x^ay^b$ if and only if $u^{1/a} = x y^{b/a}$, for $u \in
(0,1)$, it suffices to consider $a=1$.  By looking at $C_1 =
[0,\frac 13] \cup [\frac23, 1]$, it is not hard to see that if $f(x,y)
= x y^t$, $t \ge 1$, and $(\frac 23)^{1+t} > \frac 13$, then
$f(C_1^2)$ is already missing an interval from $[0,1]$. This condition
occurs when $t < \frac{\log 2}{\log 3/2} \approx 1.7095$.
\end{remark}

\begin{remark}
By taking logarithms, we can convert the question about products of
elements of $C$ into a question about sums. (We can omit the point $0$
since its multiplicative behavior is trivial.) Of course, the
underlying set is no longer the standard self-similar Cantor set but
is a more general (``non-linear'') closed subset of $\rr$. Some
conclusions about the number of factors needed to recover all of
$[0,1]$ can be obtained from the general results in the papers of
Cabrelli--Hare--Molter \cite{chm:sums}, \cite{chm:sums2}, but it does
not appear that one obtains the precise conclusion of part 2 of Theorem \ref{maintheorem} in this fashion.
\end{remark}

\subsection{Division}\label{subsec:divide}

In this section, we complete our arithmetic discussion by considering quotients.

\begin{theorem}\label{div}
\begin{equation}\label{E:div}
\left\{ \frac uv: u, v \in C\right\} = \bigcup_{m = -\infty}^{\infty}
\left[\frac 23 \cdot 3^m ,  \frac 32 \cdot 3^m\right].
\end{equation}
\end{theorem}
\begin{proof}

As with multiplication, it suffices to consider $\widetilde C$.

\begin{lemma}
Theorem \ref{div} is implied by the identity
\begin{equation}\label{E:div-c-tilde}
\left\{ \frac uv: u, v \in \widetilde C\right\} = \left[\frac 23 ,
  \frac 32 \right]\,.
\end{equation}
\end{lemma}

\begin{proof}
Write $u, v \in C$ as $u = 3^{-s}\tilde u$, $v =  3^{-t}\tilde v$ for
integers $c,d \ge 0$ and $\tilde u, \tilde v \in \widetilde C$. Then
$u/v = 3^{d-c}\tilde u/\tilde v$, where $m=d-c$ can attain any integer value.
\end{proof}

We now prove \eqref{E:div-c-tilde}. Consider $\widetilde C_1 = [\frac 23,1]$ and apply Lemma
\ref{lemma2}. Clearly, $\{\frac uv: u, v \in \widetilde C_1\} = [\frac 23, \frac 32]$.
Consider two intervals in $\widetilde C_n$,  $I_1 = [a,a+3t]$ and
$I_2=[b,b+3t]$. { These intervals are either identical or disjoint.} Since $x = \frac uv$ implies $\frac 1x = \frac vu$, there is no harm in assuming $a \le
b$, and either $I_1 = I_2$ and $a=b$, or the intervals are disjoint and $a + 3t
\le b$. The quotients  from these intervals will lie in
\[
J_0:= \left[ \frac a{b+3t}, \frac {a+3t}b \right]:= [r_0,s_0].
\]
Since $\ddot I_1 = [a,a+t] \cup [a+2t,a+3t]$
and  $\ddot I_2 = [b,b+t] \cup [b+2t,b+3t]$, we obtain four subintervals
\[
\begin{gathered}
J_1 =\left[ \frac a{b+3t}, \frac {a+t}{b+2t} \right] = [r_1,s_1], \\
J_2 =\left[ \frac a{b+t}, \frac {a+t}{b} \right] = [r_2,s_2], \\
J_3 =\left[ \frac {a+2t}{b+3t}, \frac {a+3t}{b+2t} \right] = [r_3,s_3], \\
J_4 =\left[ \frac {a+2t}{b+t}, \frac {a+3t}{b} \right] = [r_4,s_4]. \\
\end{gathered}
\]
We need to see how $J_0 = J_1 \cup J_2 \cup J_3 \cup J_4$. There are
two cases, depending on whether $a=b$ or $a<b$.

We first record some algebraic relations. We have
$r_1 = r_0$ and $s_4 = s_0$, and, evidently,
$r_1 < r_2$, $s_1 < s_2$, $r_3 < r_4$, $s_3 < s_4$. Further,
\[
\begin{gathered}
r_3 - r_2 =  \frac {a+2t}{b+3t} -  \frac a{b+t} =
\frac{2t(b-a+t)}{(b+t)(b+3t)}, \\
s_3 - s_2 = \frac {a+3t}{b+2t} -  \frac {a+t}{b} =
\frac{2t(b-a-t)}{b(b+2t)}, \\
s_1 - r_2 =  \frac {a+t}{b+2t} -  \frac a{b+t} =
\frac{t(b-a+t)}{(b+t)(b+2t)}, \\
s_2 - r_3 = \frac {a+t}{b} -  \frac {a+2t}{b+3t} = \frac{t(3a + 3t -
  b)}{b(b+3t)} \ge \frac {t(3\cdot\frac23 + 0 - 1)}{b(b+3t)} > 0, \\
s_3 - r_4 =  \frac {a+3t}{b+2t} - \frac {a+2t}{b+t} =
\frac{t(b-a-t)}{(b+t)(b+2t)}.
\end{gathered}
\]

Suppose first that $a < b$, so $ a+3t < b$. Then each of the
differences above is positive, so $r_1 < r_2 < r_3 < r_4$ and
$s_1 < s_2 < s_3 < s_4$; further, the intervals overlap: $s_1 > r_2$,
$s_2 > r_3$ and $s_3 > r_4$. Thus   $J_0 = J_1 \cup J_2 \cup J_3 \cup
J_4$.

If $a = b$, then $J_3 =\left[ \frac {a+2t}{a+3t}, \frac {a+3t}{a+2t}
\right] \subset \left[ \frac a{a+t}, \frac {a+t}{a} \right] = J_2$, so
we may drop $J_3$ from consideration. We have $r_1 < r_2 < r_4$ and  $s_1
< s_2 < s_4$  and need only show that $s_1 > r_2$ and $s_2 > r_4$. The
first is clear, and for the second,
\begin{equation}
s_2 - r_4 =  \frac {a+t}{a} -  \frac {a+2t}{a+t} = \frac{t^2}{a(a+t)}
> 0,
\end{equation}
so { $J_0 =  J_1 \cup J_2 \cup J_4$,} and we are done.
\end{proof}

\begin{remark}
We do not know an algorithm for expressing a feasible $u$ as a
quotient of elements { in} $C$.
\end{remark}

\subsection{Multiplication, revisited}\label{subsec:times-again}

Let $g(x,y) = xy$. As noted earlier, $g(C^2)$ is not the full interval
$[0,1]$, though $g(C^2) = \cap g(C_i^2)$ is the intersection of a
descending chain of closed sets and so is closed. In order to gain
some information about $g(C^2)$, we look carefully at how Lemma
\ref{lemma2} fails.

\begin{lemma}\label{middle}
Let $I = [a,a+3t]$ and $J = [b,b+3t]$, with { $\tfrac23 \le a \le b \le 1$,} be either
identical or disjoint intervals. Then
\begin{equation*}
\begin{gathered}
a < b \implies g(\ddot I, \ddot J) = g(I, J)\,; \\
a = b \implies g(\ddot I, \ddot I) = g(I,I) \setminus
((a+2t)^2-t^2,(a+2t)^2).
\end{gathered}
\end{equation*}
\end{lemma}

{
\begin{proof}
We have
$$
g([a,a+3t],[b,b+3t]) = [ab, ab + 3(a+b)t + 9t^2]
$$
and
\begin{equation}\label{eq:gab}
\begin{gathered}
g([a,a+t]\cup[a+2t,a+3t],[b,b+t]\cup[b+2t,b+3t]) = \\
[ab, ab + (a + b)t + t^2] \cup [ab + 2at, ab + (3a + b)t + 3t^2] \\
\cup [ab + 2b t, ab + (a + 3b)t + 3t^2] \\ \cup[ ab + (2a+2b)t +
4t^2, ab + 3(a+b)t + 9t^2].
\end{gathered}
\end{equation}
Since $a \le b$,  it follows that $ab + 2at \le  ab + (a + b)t + t^2$,
and the first two intervals coalesce into $[ab, ab + (3a + b)t + 3t^2]$.

Suppose that $a < b$, and recall that we have assumed $\tfrac23 \le a < b \le 1$.
Since $a+t \le b$, it follows that $ab+(2a+2b)t+4t^2 \le ab + (a+3b)t +3t^2$
and the last two intervals coalesce into $[ab+2bt,ab+3(a+b)t+9t^2]$. Thus the right-hand side of \eqref{eq:gab} reduces to
\begin{equation}\label{eq:gab2}
[ab, ab + (3a + b)t + 3t^2] \cup [ ab + 2bt, ab + 3(a+b)t + 9t^2]
\end{equation}
Moreover,
\begin{equation*}
ab+ (3a+b)t + 3t^2 - (ab + 2bt) = t(3a+3t - b) \ge t (3 \cdot \tfrac23 + 0 - 1) = t \ge 0,
\end{equation*}
which shows that the pair of intervals in \eqref{eq:gab2} coalesces to a single interval. This proves the first statement.

If $a = b$, then the middle two intervals in \eqref{eq:gab} are the same, and
\begin{equation*}\begin{split}
&g(([a,a+t]\cup[a+2t,a+3t])^2) \\
&\quad = [a^2, a^2+4at + 3t^2] \cup [a^2 + 4at + 4t^2, a^2 + 6at + 9t^2] \\
&\quad = [a^2,(a+3t)^2] \setminus ((a+2t)^2-t^2,(a+2t)^2)).
\end{split}\end{equation*}
\end{proof}
}

This leads to the following estimate for the Lebesgue measure of  $g(C^2)$.
\begin{theorem}\label{meas}
\[
\mu(g(C^2)) \ge \frac {17}{21}.
\]
\end{theorem}
\begin{proof}
First note that $g(\widetilde C^2) \subset
g(\widetilde C_1^2) = [\frac 49,1]$.
It follows as before
\begin{equation}
g(C^2) = \{0\} \cup \bigcup_{k=0}^\infty 3^{-k}g(\widetilde C^2),
\end{equation}
and since $\frac 13 < \frac 49$, the sets $3^{-k}g(\widetilde C^2)$
are disjoint. Therefore,
\begin{equation}
\mu(g(C^2)) = \sum_{k=0}^\infty 3^{-k}\mu(g(\widetilde C^2)) =
\tfrac 32 \mu(g(\widetilde C^2)).
\end{equation}
Since $\widetilde C_n$ consists of $2^{n-1}$ intervals of length
$3^{-n}$, it follows from Lemma \ref{middle} that
\begin{equation}\begin{split}
&\mu(g(\widetilde C_{n+1}^2)) \ge \mu(g(\widetilde C_{n}^2)) -
\frac {2^{n-1}}{3^{2n+2}} \\
&\implies \mu(g(\widetilde C^2)) \ge \left(1-\frac 49\right) - \sum_{n=1}^{\infty} \frac {2^{n-1}}{3^{2n+2}} =
\frac{34}{63} \\
&\implies \mu(g(C^2)) \ge \frac 32\cdot \frac{34}{63} = \frac {17}{21}.
\end{split}\end{equation}
\end{proof}

\begin{remark}
This argument shows that for all $m$,
\begin{equation}
\begin{gathered}
\mu(g(\widetilde C_m^2)) \ge \mu(g(\widetilde C^2))
\ge\mu(g(\widetilde C_m^2)) -  \sum_{n=m+1}^{\infty} \frac {2^{n-1}}{3^{2n+2}}.
\end{gathered}
\end{equation}
\end{remark}

The reason that Theorem \ref{meas} is only an estimate is that there is no
guarantee that intervals missing from $f(\ddot I^2)$ cannot be covered
elsewhere. The first instance in which this occurs is for $n=4$: one
of the intervals in $\widetilde C_4$ is $I_0 = [\frac {62}{81},\frac
{63}{81}]$ = $[.2022_3,.21_3]$. By Lemma \ref{middle},
\begin{equation}
(\tfrac{188^2-1}{243^2}, \tfrac{188^2}{243^2})  =
 (\tfrac{35343}{59049},\tfrac{35344}{59049})
 \approx (.5985368,.5985537) \notin f(\ddot I_0^2).
\end{equation}
However, $\widetilde C_5$ contains the intervals
\begin{equation}
 J_1 = [\tfrac {162}{243}, \tfrac {163}{243}] = [.2_3,.20001_3], \quad
 J_2 = [\tfrac  {216}{243}, \tfrac  {217}{243}] = [.22_3,.22001_3]
\end{equation}
and
 \[
 J_1J_2  = [\tfrac{34992}{59049},\tfrac{35371}{59049}] \approx
 (.5925926,.59901099)
 \]
covers the otherwise-missing interval.
A more detailed {\it Mathematica} computation, using $m=11$, gives the first
eight decimal digits for $\mu(g(C^2))$:
\begin{equation}
\mu(g(C^2)) = .80955358\dots \approx \frac{17}{21} + 2.97 \times 10^{-5}.
\end{equation}

\section{Final remarks.}\label{sec:final-remarks}

As mentioned in the introduction, this paper is part of a larger project. { We discuss a few results from this project whose proofs will appear elsewhere, written by various combinations of the authors and their students.}

\subsection{Self-similar Cantor sets}

The Cantor set easily generalizes to sets defined with different
``middle-fractions'' removed.
Consider  the self-similar Cantor set $D^{(t)}$ obtained as the
invariant set for the pair of contractive mappings
$$
f_1(x) = tx, \qquad f_2(x) = tx+(1-t)
$$
acting on the real line. Thus
$$
D^{(t)} = \bigcap_{n\ge 0} D_n^{(t)},
$$
where, for each $n\ge 0$, $D_n^{(t)}$ is a union of $2^n$ intervals,
each of length $t^n$, contained in $[0,1]$.
For instance,
$$
D_1^{(t)} = [0,t] \cup [1-t,1],
$$
$$
D_2^{(t)} = [0,t^2]\cup[t(1-t),t]\cup[1-t,1-t+t^2]\cup[1-t^2,1]\,.
$$
For an integer $m\ge 2$, let $t_m$ be the unique solution to
$(1-t)^m=t$ in $[0,1]$. Then
$$
\cdots < t_4 < t_3 < t_2 < \frac12
$$
and $t_m \to 0$ as $m\to\infty$. Numerical values are
\begin{eqnarray*}
&t_2 = \tfrac{3-\sqrt5}2 \approx 0.381966 \dots \\
&t_3 \approx 0.317672 \dots \\
&t_4 \approx 0.275508 \dots \\
&t_5 \approx 0.245122 \dots \\
&t_6 \approx 0.22191 \dots \\
&t_7 \approx 0.203456 \dots
\end{eqnarray*}
If we choose $t$ such that $t_m \le t < t_{m-1}$ (that is,  $(1-t)^m \le t < (1-t)^{m-1}$) and let $g_n(x_1,\dots,x_n) = x_1x_2\cdots x_n$, then
\[
g_m((D^{(t)})^m) = [0,1],
\]
but $g_{m-1}((D^{(t)})^{m-1})$ has Lebesgue measure strictly less than one.
In particular, if a Cantor set $D$ in which a middle fraction $\la$ is
taken with $\la \le 1 - 2t_2 = \sqrt 5 - 2 \approx .23607$, then every
element in $[0,1]$ can be written as a product of two elements of $D$.

\subsection{Number Theory}
{ There is much more to be said about the representation of specific numbers in $C/C$. For example, if  $v = 2u$ for $u,v \in C$, then \begin{equation*} u = \frac 1{3^n},\ v  = \frac 2{3^n},\quad n \ge 1; \end{equation*} if $v = 11u$ for $u,v \in C$, then \begin{equation*} u = \frac 1{4\cdot 3^n},\ v = \frac {11}{4\cdot 3^n},\quad n \ge 1.
\end{equation*}
By contrast, if $v = 4u$ for $u,v \in C$, then there exists a finite or infinite sequence of integers $(n_k)$ with $n_1 \ge 2$ and  $n_{k+1}-n_k \ge 2$, so that \begin{equation*} u = \sum_k \frac 2{3^{n_k}},\qquad v = \sum_k\left(\frac 2{3^{n_k-1}}
   + \frac 2{3^{n_k}}\right).
\end{equation*}
The proof of the second result is trickier than the other two.}


We also mention a conjecture for which there is strong numerical evidence.

\begin{conjecture}
Every $u \in [0,1]$ can be written as $x_1^2 + x_2^2 + x_3^2 + x_4^2$,
$x_i \in C$.
\end{conjecture}

We need a minimum of four squares, since $(\frac 13)^2 + (\frac
13)^2+(\frac 13)^2 < (\frac 23)^2$, so the open interval $(\frac 13,
\frac 49)$ will be missing from the sum of three squares.

\section{Acknowledgments}
The authors wish to thank the referees and editors for their rapid, sympathetic
and extremely useful suggestions for improving the manuscript. BR
wants to thank Prof. W. A. J. Luxemburg's Math 108 at Caltech in 1970--1971
for introducing him to the beauties of analysis.

\end{document}